\documentclass[11pt]{article}
\usepackage[english]{babel}
\usepackage{amsfonts,amssymb,amsmath,amsthm,latexsym}
\usepackage[usenames]{color}
\usepackage{inputenc}
\usepackage[normalem]{ulem}
\usepackage{enumerate}
\usepackage{tikz}
\usetikzlibrary{shapes}
\usetikzlibrary{plotmarks}

\newcommand{\Z}{\mathbb{Z}}
\newcommand{\N}{\mathbb{N}}
\newcommand{\Q}{\mathbb{Q}}
\newcommand{\C}{\mathbb{C}}

\newcommand{\tx}{\textrm}
\newcommand{\w}{\upsilon_{n,m}}
\newcommand{\ord}{\hbox{\rm ord}}
\newcommand{\mult}{\hbox{\rm mult$\,$}}

\newcommand{\supp}{\hbox{\rm supp$\,$}}

\newcommand{\NP}{\mathcal{NP}}

\newcommand{\dd}{{\hbox{\rm d}}}

\newtheorem{teorema}{Theorem}[section]
\newtheorem{lema}[teorema]{Lemma}
\newtheorem{defi}[teorema]{Definition}

\newtheorem{ejemplo}[teorema]{Example}
\newtheorem{coro}[teorema]{Corollary}
\newtheorem{prop}[teorema]{Proposition}
\newtheorem{nota}[teorema]{Remark}

\title{On characterizations of nondicritical generalized curve foliations \footnotetext{
     \noindent   \begin{minipage}[t]{4in}
       {\small
       2000 {\it Mathematics Subject Classification:\/} Primary32S65;
       Secondary 32S05.\\
       Key words and phrases: Generalized curve foliation, Weierstrass form, Plane singular analytic curve, Second type foliation.\\
       The first-named author was partially supported by the Spanish Projects
    PID2019-105896GB-I00 and the second-named author was partially supported by CNPq-Brazil.}
       \end{minipage}}}

\author{Evelia R.\ Garc\'{\i}a Barroso, Marcelo E. Hernandes and M. Fernando Hern\'andez Iglesias}

\begin{document}
\maketitle

\begin{abstract}

We characterize nondicrital generalized curve foliations with fixed reduced separatrix. Moreover, we give suficient conditions when a plane analytic curve  is  its reduced separatrix.
For that, we introduce a distinguished expression
 for a given  1-form, called {\it Weierstrass form}. Then, using Weierstrass forms, we characterize the nondicritical generalized curve foliations: first, for foliations with monomial separatrix using toric resolution; second, for
 foliations with reduced separatrix, using the $GSV$-index. In this last case the characterization, which is our main result, could be interpreted in function of  a polar of the foliation and a polar of its reduced separatrix. 

\end{abstract}

\section{Introduction}

 In the study of holomorphic foliations to determine topological information is an important and a non trivial question. A special class of 
 	foliations gives some information about this problem: the generalized curve  foliations.  A foliation is said generalized curve if no saddle-nodes appear in its desingularization process.
Moreover, a nondicrital generalized curve foliation and the union of its separatrices have the same process of reduction of singularities (see [Cam-LN-S], Theorem 2).
	
In this paper, we consider this class of foliations with the aim to present some new characterizations of nondicritical generalized curve foliation by means the $1$-form that defines it.

First, we characterize, in Theorem \ref{monomial}, the nondicrital generalized curved foliations with monomial separatrix using toric resolution, the prenormal form given by Loray in \cite[page 157]{Loray} and the notion of \emph{weighted order} associated with a foliation and a curve (see Definition \ref{worder}). This allows us to give, in Corollary \ref{monomial second type}, the condition in order that a foliation of second type becomes a generalized curve foliation following \cite[Theorem 1.2 (b)]{FS-GB-SM}. Second, using the Weierstrass division of power series, we introduce, in Defi\-nition \ref{Wform},  the notion of \emph{Weierstrass form} of any $1$-form with respect to any polynomial in one variable and coefficients in the complex power ring in one variable with a unit as principal coefficient. We can consider
 the Weierstrass form as  a generalization, to any 1-form, of the prenormal form given by Loray  for foliations with monomial separatrix. \\

In \cite{Brunella}, Brunella stablished  relations among the Baum-Bott, Camacho-Sad and  G\'omez-Mont- Seade-Verjovsky  indices associated to one foliation.  He proved that if the foliation is generalized curve then the Baum-Bott and Camacho-Sad indices are equal and the G\'omez-Mont-Seade-Verjovsky index is zero. This paper goes in this direction.\\

Our  main result  is Theorem \ref{general}, where we characterize the nondicritical generalized curve foliations using the \emph{GSV-index} and the notion of Weierstrass form associated with a $1$-form defining a foliation.
This characterization is given in terms of  intersection numbers and it could be interpreted in the language of polar curves of the foliation and its  union of separatrices.

We precise,  in Proposition \ref{prop:irred}, this characterization for foliations with a single separatrix  and for foliations with a single separatrix of genus 1 in Corollary \ref{coro:genus1}, this last one in terms of the weighted order.

The Weierstrass form has been very useful throughout the development of this work and we consider that it may be interesting in the future for foliation studies. Example of this is Proposition \ref{Fol(f)}, where given $f\in \C\{x\}[y]$ with
$\ord f=\deg_y(f)$ and principal coefficient a unit in $\C\{x\}$,
and using the notion of Weierstrass form of a 1-form ${W}$ (defining a non-dicritical foliation) with respect to $f$, we give suficient conditions when $f$ is the equation of the union of separatrices of  the foliation determined by $W$.

\section{Preliminaries}
\label{preliminaries}
\medskip

Let $f(x,y)\in \mathbb C\{x,y\}$ be a non unit power series,  where $\mathbb C\{x,y\}$ is the set of convergent power series. A {\bf plane curve} $C:\{f(x,y)=0\}$ is by definition the zeroes set determined by $f(x,y)\in \mathbb C\{x,y\}$. The curve $C$ is irreducible (respectively reduced) if $f$ is irreducible in $\mathbb C\{x,y\}$ (respectively $f$ has no multiple factors). An irreducible plane curve is called {\bf branch}. The {\bf multiplicity} of $C$, denoted by $\mult C$,  is by definition the order of the power series $f(x,y)$, that is $\mult C=\ord f$. Remember that the order of $f$ is the minimum of degrees of terms of $f$.

\medskip

 Consider a branch $C:\{f(x,y)=0\}$ of multiplicity $n$. After a change of coordinates we can suppose that $x=0$ is transversal (not tangent) to $C$ at $0$. From Newton's Theorem,  $C$ admits an expansion with rational exponents $y(x^{1/n})$, such that $f(x,y(x^{1/n}))=0$. According to Puiseux, the branch $C$ admits $n$ different expansions $\{y_i(x^{1/n})\}_{1\leq i\leq n}$, where $y_i(x^{1/n})=y(\varepsilon_i x^{1/n})$ and $\{\varepsilon_i\}_{1\leq i\leq n}$ are the $n$th roots of unity in $\C$. We can write $f$, up to product by a unit in $\C\{x,y\}$, as the product

 \[f(x,y)=\prod_{i=1}^n\left(y-y_i(x^{1/n})\right).\]

The expansions $\{y_i(x^{1/n})\}_{1\leq i\leq n}$ are called {\bf Newton-Puiseux roots} of the branch $C$ (or equivalently of $f$). Any Newton-Puiseux root $y_i(x^{1/n})$ is of the form

\[
y_i(x^{1/n})=\sum_{j\geq n}a_j^{(i)}x^{j/n},
\]

\noindent where $j\in \N$, $a_j^{(i)}\in \C$ and it determines a {\bf parametrization} of $C$ as follows
\[
(x_i(t),y_i(t))=\left(t^n,\sum_{j\geq n}a_j^{(i)}t^{j}\right).
\]

It is well-known that if the multiplicity of a branch is bigger than one, then there exist $g\in \N\backslash\{0\}$ and positive integers $\beta_0=n$ and
$\beta_k=\min\{k\,:\, a^{(i)}_k\neq 0\ \mbox{and}\gcd(n,\beta_1,\beta_2,\ldots,\beta_{k-1})\;\hbox{\rm does not divide } k\}$ for $1\leq k\leq g$. In the sequel we put  $e_k=\gcd(\beta_0,\beta_1,\beta_2,\ldots,\beta_k)$ for $0\leq k\leq g$. We get $e_0=n>e_1>\cdots >e_g=1$. Set $n_k:=\frac{e_{k-1}}{e_k}$. In particular, $n=\beta_0=n_1\cdots n_g$. 
 
The sequence $(\beta_0,\beta_1,\beta_2,\ldots,\beta_g)$ is called the {\bf Puiseux characteristic exponents} of the branch $C$ and the number $g$ is called the {\bf genus} of the branch $C$. We denote by  $K(\beta_0,\beta_1,\beta_2,\ldots, \beta_g) $ the set of plane branches with Puiseux characteristic exponents $(\beta_0,\beta_1,\beta_2,\ldots,\beta_g)$. It is well-known, after Brauner-Zariski, that the characteristic exponents determine the topological class of the branch $C$.

\medskip

Let $h_1(x,y), h_2(x,y)\in \C\{x,y\}$ be two power series and $I=(h_{1},h_{2})$ the ideal generated by $h_1,h_2\in \C\{x,y\}$. The {\bf intersection number} of the curves 
$C_i: \{h_i(x,y)=0\}$, $1\leq i\leq 2$,  is 
$i_0(h_1,h_2):=\dim_{\mathbb C}\mathbb{C}\{x,y\}/I.$

\medskip

Given $f\in K(\beta_0,\beta_1,\beta_2,\ldots ,\beta_g)$ the semigroup $\Gamma(f)$ associated to $f$ is 
\[
\Gamma(f)=\{i_0(f,h)\,:\,\ h\in\mathbb{C}\{x,y\}\setminus (f)\}.
\]

\medskip

The semigroup $\Gamma(f)$ admits a unique minimal system of generators, given by $\{v_0,v_1,v_2,\ldots ,v_g\}$, that is, $\Gamma(f)=\langle v_0,v_1,v_2,\ldots ,v_g\rangle=\left \{\sum_{i=0}^{g}\alpha_iv_i\,:\, \alpha_i\in\mathbb{N}\right\}$. 
It is a well know fact that the semigroup $\Gamma(f)$ and the characteristic exponents are mutually determined.
Moreover, we can obtain the minimal system of generators by the Puiseux characteristic exponents using the relations (see \cite{Zariski}, Theorem 3.9):
\begin{equation}
\label{exponentes Puiseux}
v_0=n=\beta_0,\ \ \ v_1=m=\beta_1,\ \ \ v_{i+1}=n_iv_i+\beta_{i+1}-\beta_{i}\ \ \mbox{for}\ i=1,\ldots ,g-1.
\end{equation}

Observe that $e_{i}=\gcd(v_{0},\ldots,v_{i})$ for $0\leq i\leq g$. 

For any reduced plane curve (not necessary irreducible) $C:\{f(x,y)=0\}$, an  important topological invariant, useful in this paper, is the {\bf Milnor number}, that is
the intersection number $\mu(f):=i_0\left(f_{x},f_{y}\right)$, where $f_{y}$ (respectively $f_{x}$) denotes the partial derivative of $f$ with respect to $y$ (respectively $x$).
\medskip

By Teissier's Lemma (see \cite[Chapter II, Proposition 1.2]{Teissier}) we have
\begin{equation}
\label{LTeissier}
i_{0}(f,f_{y})=\mu(f)+i_{0}(f,x)-1.
\end{equation}

We also have
\begin{equation}
\label{LTeissier2}
i_{0}(f,f_{x})=\mu(f)+i_{0}(f,y)-1.
\end{equation}

For more details on plane curves see  for example \cite{Hefez} or \cite{Wall}.

\medskip

Let $\Omega^{1}_{\mathbb{C}^2,0}:=\mathbb{C}\{x,y\}\dd x+\mathbb{C}\{x,y\}\dd y$
be the $\mathbb{C}\{x,y\}$-module of holomorphic 1-form.

\medskip

A {\bf holomorphic foliation singular at the origin} is defined, in a neighbourhood of the origin, by an equation ${\mathcal F}_W: \{W=0\},$ where $W$ is a 1-form
$W=A(x,y)\dd x+B(x,y)\dd y,$ with $A,B \in \C\{x,y\}$ without common factors and such that $A(0,0)=B(0,0)=0$. The {\bf polar} of ${\mathcal F}_W: \{W=0\}$ with respect the direction $(b:-a)\in \mathbb P^{1}$ is the curve $aA(x,y)+bB(x,y)=0$.

\medskip

The {\bf multiplicity} of $W$ is 
$\mult(W):=\min\{\ord(A(x,y)),\ord(B(x,y))\}$. More precisely if $\mult(W)=n_0$ then we can write
$W= \sum_{i+j \geq n_{0} } a_{i,j}x^iy^j\dd x + \sum_{i+j \geq n_{0}} b_{i,j}x^iy^j\dd y$, and for some $i,j$ such that  $i+j=n_{0}$ we get $a_{i,j} \neq 0$ or $b_{i,j}\neq 0$. The {\bf multiplicity} of the foliation ${\mathcal F}_W: \{W=0\}$ is by definition the multiplicity of the 1-form $W$.

\medskip
Let ${\mathcal F}_W$ be a singular foliation at the origin. The {\bf Jacobian matrix} of the linear part of  $W$ is the matrix
\[
\left (
\begin{array}{rr}
-B_x(0,0) & -B_{y}(0,0) \\
A_{x}(0,0) & A_{y}(0,0)
\end{array}
\right ).
\]
Denote by $\lambda_1,\lambda_2$ its complex eigenvalues. We say that the origin is an {\bf irreducible  singularity} of ${\mathcal F}_W$ if one of the following conditions is satisfied:

\begin{enumerate}
    \item $\lambda_1\lambda_2\neq 0$ and $\frac{\lambda_1}{\lambda_2} \notin \Q^{+}$,
    \item $\lambda_1 =0 $ and $\lambda_2 \neq 0$; or $\lambda_2=0$ and $\lambda_1\neq 0$.
\end{enumerate}

If the first condition holds we say that the origin is a {\bf simple} or {\bf non-degenerate singularity} of ${\mathcal F}_W$. Nevertheless if the second condition holds then we say that the origin is a {\bf saddle-node singularity} of ${\mathcal F}_W$.

\begin{ejemplo}\label{ej1}
If $W=y^2\dd x+x\dd y$ then the Jacobian matrix of the linear part of  $W$ is $\left (
\begin{array}{rc}
-1 & 0 \\
0 & 0
\end{array}
\right )$.
Its eigenvalues are $-1$ y $0$, so the origin is a  saddle-node singularity of ${\mathcal F}_W\,$.
\end{ejemplo}

\medskip

Let $\Pi:(M,D)\longrightarrow (\C^2,0)$ the {\bf process of reduction of singularities} of the foliation ${\mathcal F}_{W}$  (see  \cite{Seidenberg}), given by a finite composition of quadratic transformations (blow-ups) having $D=\Pi^{-1}(0)$ as the \emph{exceptional divisor}, which is a finite union of projective lines with normal crossings; and where $M$ is a non-singular analytic manifold. 
There is a 1-form $W'$ such that $\Pi^{-1}(0)\cdot W'=\Pi^*(W):=\Pi^*(\mathcal{F}_W)$. The foliation ${\mathcal F}_{W'}$ is called the {\bf strict transform} of ${\mathcal F}_W$.

\medskip

A foliation  is {\bf generalized curve} if no saddle-nodes appear in its desingularization.

\medskip

A branch $C:\{f(x,y)=0\}$ is a {\bf separatrix} of the foliation ${\mathcal F}_W: \{W=0\}$ if  $f$ divides $W\wedge df$. This is equivalent to the existence of a parametrization  $(x(t),y(t))$  of the branch $C$ such that $A(x(t),y(t))\dd x(t)+B(x(t),y(t))\dd y(t)=0$. In \cite{C-S}, Camacho and Sad prove that any singular foliation in $(\C^2,0)$ admits a separatrix. Their proof is based on the desingularization process of the foliation.  In Example \ref{ej1}  the branches $x=0$ and $y=0$ are separatrices of  ${\mathcal F}_W$. According to the number of separatrices that pass through the singular point of the foliation we can classify the foliations in {\bf dicritical} in the case that we have infinite separatrices or  {\bf nondicritical} when we have a finite number of separatrices.
As a consequence of \cite[Theorem 1]{C-LN-S} we get
\begin{teorema}
\label{comparando multiplicidades}
Let  ${\mathcal F}_W$ be a nondicritical foliation with  union of separatrices $\{f(x,y)=0\}$. Then
$\mult(W)\geq \ord(f)-1$.
\end{teorema}

\begin{ejemplo} If  $W_1=nx\dd y-my\dd x$ with $n$ and $m$ coprime positive integers, then the branches $\{x=0\},\{y=0\}$ and $\{y^n-cx^m=0\},$  $c \neq 0$ are separatrices
of $W_1=0,$ so the foliation ${\mathcal F}_{W_1}$ is dicritical. Nevertheless if
 $W_2=x\dd y+y\dd x$  then the foliation ${\mathcal F}_{W_2}$ is nondicritical since its  only separatrices are $\{x=0\}$ and $\{y=0\}$.
 \end{ejemplo}

Let $C_i:\{f_i(x,y)=0\}$, $1\leq i \leq r$, be the set of different separatrices of a nondicritical foliation ${\mathcal F}_W: \{W=0\}$ and put $f(x,y):=f_1(x,y)\cdots f_r(x,y)$.  The reduced curve $C:\{f(x,y)=0\}$ will be called the {\bf union of the separatrices} of the nondicritical foliation ${\mathcal F}_W$.\\

Other important notion we will use is the Newton polygon that we introduce in the sequel.
Let $S\subset \mathbb N^2$.  We consider the convex hull
$\mathtt {conv}(S)$ of the Minkowski sum $S+\mathbb R^2_{\geq 0}$, where $\mathbb R_{\geq 0}$ denote the non-negative real numbers. By definition, the {\bf Newton polygon} of $S$, denoted by $\mathcal{NP}(S)$, is
the union of the compact edges of the boundary of $\mathtt {conv}(S)$.  

\medskip
The {\bf support of a power series} $f(x,y)=\sum_{i,j}a_{i,j}x^iy^j$ is $\supp f:=\{(i,j)\in \mathbb N^2\;:\; a_{i,j}\neq 0\}$. The {\bf support of a foliation} ${\mathcal F}_W: \{W=0\}$ , where $W=A(x,y)\dd x+B(x,y)\dd y$, is the union of the supports of $x\cdot A(x.y)$ and $y\cdot B(x,y)$, and we will denote it $\supp W$, that is $\supp W=\supp (x\cdot A(x,y)) \cup \supp (y\cdot B(x,y))$. The {\bf Newton polygon of a power series} $f(x,y)\in \mathbb C\{x,y\}$ is by definition the Newton polygon of $\supp f$ and we will denote it $\mathcal{NP}(f)$. The {\bf Newton polygon of a foliation} ${\mathcal F}_W$ is by definition the Newton polygon of $\supp W$, and we will denote it $\mathcal{NP}(W)$. Observe that the Newton polygon depends on coordinates. Remark that $\mathcal{NP}(u\cdot f)=\mathcal{NP}(f)$
for any $u,f\in \mathbb C\{x,y\}$,  where $u$ is a unit. Hence, we can define the {\bf Newton polygon} of the curve $C:\{f(x,y)=0\}$ as the Newton polygon of any of its equations.

\begin{prop}(\cite[Proposition 3.8]{PR})
\label{Newton}
Let $W=A(x,y)\dd x+B(x,y)\dd y$ be a 1-form. If ${\mathcal F}_W$ is a nondicritical generalized curve foliation and $C:\{f(x,y)=0\}$ is  its union of separatrices then the Newton polygons of ${\mathcal F}_W$ and $C$ are the same.
\end{prop}

\medskip

Saravia, in her PhD thesis \cite{Nancy} (see also \cite[Example 3.3]{FS-GB-SM}),  shows that the foliation ${\mathcal F}_W$, where $W=((b-1)xy-y^3)\dd x+(xy-bx^2+xy^2)\dd y$, with $b\not\in \Q$, is not a generalized curve foliation but its Newton polygon equals to the Newton polygon of its  union of separatrices $f(x,y)=xy(x-y)$.

\medskip

Milnor number can be also defined in the foliation context.
Let $A(x,y),B(x,y)\in \C\{x,y\}$ with $A(0,0)=B(0,0)=0$. Let $W=A(x,y)\dd x+B(x,y)\dd y$ be a 1-form and consider the singular foliation at the origin ${\mathcal F}_W$.  The {\bf Milnor number at the origin of  the foliation} ${\mathcal F}_W$  is $\mu({\mathcal F}_W):=i_{0}(A(x,y),B(x,y))$. \\

Following  \cite[Theorem 4]{C-LN-S} we have the next characterization of nondicritical generalized curve foliations using Milnor numbers:

\begin{teorema}\label{curvagene}
Let $W=A(x,y)\dd x+B(x,y)\dd y$  be a 1-form. Suppose that  ${\mathcal F}_W$ is a nondicritical generalized curve foliation and $C:\{f(x,y)=0\}$  is its  union of separatrices. Then
 $\mu({\mathcal F}_W)\geq \mu(f)$. The equality holds  if and only if
 ${\mathcal F}_W$ is a generalized curve foliation.
\end{teorema}

Let $C:\{f(x,y)=0\}$ be a reduced plane curve. In what follows we consider:

\smallskip

$\bullet \;$ $[Fol(f)]$ the set of all 1-forms $W$, defining a nondicritical foliation ${\mathcal F}_{W}$, such that  $f$ divides $W\wedge \dd f$ .

\medskip

$\bullet \;$ $Fol(f)$ the set of nondicrital foliations defined by elements of $[Fol(f)]$ which union of  separatrices is $C:\{f(x,y)=0\}$.
 
\medskip
	
	In this paper we will present new characterizations for generalized curve foliations.

\medskip

\section{Characterization of generalized curve foliations with monomial separatrix}
\label{sect: char monomial}
 In this section we characterize the nondicritical generalized curve foliations with monomial separatrix, by means of the weighted order defined as following
 
 \begin{defi}
\label{worder}
Let $p,q \in \Z^{+}$. The {\bf weighted order} $\upsilon_{p,q}$ for power series and differential forms is:
\[\upsilon_{p,q}\big(\sum_{i,j}a_{i,j}x^iy^j)=\min\{ip+jq\;:\; a_{i,j}\neq 0\},\]
\noindent and
\[\upsilon_{p,q}\big(\sum_{i,j} A_{i,j}x^{i-1}y^j \dd x +\sum_{i,j} B_{i,j}x^iy^{j-1}\dd y) =\min \{ip+jq\;:\;A_{i,j}\neq 0 \,\, \tx{or}\,\,B_{i,j}\neq 0  \}.\]
\end{defi}

\medskip

Let $f\in \C\{x,y\}$ and $n,m$ positive integers not necessarily coprime. For  $f=y^n-x^m$, Frank Loray  obtained a {\bf prenormal form} of an arbitrary element ${\mathcal F}_W\in {\mathcal Fol}(f)$.
\begin{teorema} (\cite[page 157]{Loray})
\label{Loray}
If $f=y^n-x^m$ and  ${\mathcal F}_W\in Fol(f)$ then
\begin{equation}
\label{eq:L}
W=\dd f+ \sum_{{\begin{array}{c}0\leq i \leq m-2\\ 0\leq j \leq n-2 \end{array}}}P_{i,j}(f)x^iy^j(nx\dd y-my\dd x),
\end{equation}
for some  $P_{i,j}(z) \in \C\{z\}$.
\end{teorema}

For $f(x,y)=y^n-x^m$ with $0<n\leq m$ not necessary coprime,  we will give a characterization when the foliation   ${\mathcal F}_W \in Fol(f)$ is generalized curve.\\

The 1-form $W$ in $(\ref{eq:L})$ can be rewritten as

\begin{equation}\label{quasihomogenea}
W=\dd f+(\Delta(x,y)+f(x,y)g(x,y))(nx\dd y-my\dd x),
\end{equation}

where $g\in \C\{x,y\}$ and $\Delta(x,y)=\displaystyle\sum_{{\begin{array}{c}0\leq i \leq m-2\\ 0\leq j \leq n-2 \end{array}}}a_{i,j}x^iy^j.$

\medskip

Let $\gcd(n,m)=d$. Since $m\geq n$ there are integers  $p,q \in \mathbb Z^{+}$ with $mp-nq=d$.

\begin{lema}
\label{singular points} Let  ${\mathcal F}_W\in Fol(f)$ be a foliation defined by $W$ as in (\ref{quasihomogenea}). If $\w(\Delta(x,y))\geq mn-n-m$ then the toric morphism

\begin{equation*}
\label{toric}
\Pi(u,v)=\left (u^pv^{\frac{n}{d}},u^qv^{\frac{m}{d}}\right)
\end{equation*}
 desingularizes  ${\mathcal F}_W$ and the singular points of its strict transform are $(0,0)$ and $(\xi,0)$, where $\xi^d=1$. Moreover the Jacobian matrix at the point $(0,0)$ is 

\begin{equation}
\label{eq:0}
\left (\begin{array}{cc}-\frac{mn}{d} & 0 \\
0 & nq
\end{array} \right),
\end{equation}

and the Jacobian matrix at the point $(\xi,0)$  is:
\begin{equation}
\label{eq:singp}
\left (
\begin{array}{cc}
mn & 0 \\
* & -d(1+\displaystyle \sum_{r,s}a_{r,s}\xi^{p(r+1)+q(s+1)-qn})
\end{array} \right),
\end{equation}
where $r,s$ 	verify 
 $m(r+1)+n(s+1)=nm$. 

\end{lema}

\begin{proof}
Notice that $\Pi^{*}W=u^{qn-1}v^{\frac{mm}{d}-1}\cdot W'$, where

\begin{eqnarray}
\label{str}
W'&=&[(nq-mpu^{d})+(nq-mp)S
+h(u,v)(nqu^{p+q}v^{\frac{n+m}{d}}-pmu^{d+p+q}v^{\frac{n+m}{d}})]vdu \nonumber \\
&+& \left[\frac{mn}{d}(1-u^d)u+\frac{mnh}{d}(u^{p+q+1}v^{\frac{m+n}{d}}-
u^{p+q+d+2}v^{\frac{m+n}{d}} ) \right] dv,
\end{eqnarray}

 with $S= \sum a_{i,j} u^{p(i+1)+q(j+1)-qn}v^{\frac{m(j+1)+n(i+1)-mn}{d}}$ and $h(u,v)\in \mathbb C\{u,v\}$. By hypothesis, if $(i,j)\in \supp \Delta$ then $ni+jm\geq mn-m-n$, and we have $\tx{ord}_{v}(S)\geq 0$. On the other hand we claim that   $\ord_{u}(S) > 0$, that is $p(i+1)+q(j+1)> qn$. In effect, from  $ni+jm\geq mn-n-m$\, and $mp-nq=d$ we get
  $p(i+1)+q(j+1) = \frac{ pn(i+1)+pm(j+1)}{n}-\frac{d(j+1)}{n} \geq \frac{pmn-d(j+1)}{n}$.  Since $0\leq j\leq n-2$ we have $\frac{pmn-d(j+1)}{n}>\frac{pnm-dn}{n}$, so $p(i+1)+q(j+1)>pm-d=nq$.\\ 
  
Therefore the singular points of ${\mathcal F}_{W'}$ are $(0,0)$ and $(\xi,0)$, with $\xi^d=1$. The point $(0,0)$ corresponds to the intersection of the exceptional divisors, the other points correspond to the irreducible components of $y^n-x^m$. 
 
Rewrite \eqref{str} as $W'=A(u,v)\dd u+B(u,v)\dd v$. We get $B_{u}(0,0)=\frac{mn}{d}$, $B_{v}(0,0)=A_{u}(0,0)=0$ and $A_{v}(0,0)=nq$, so the Jacobian matrix at $(0,0)$ is \eqref{eq:0} and the origin is not a saddle-node. On the other hand $B_{u}(\xi,0)=\frac{mn}{d}$, $B_{v}(\xi,0)=A_{u}(\xi,0)=0$ and $A_{v}(\xi,0)=1+\displaystyle \sum_{m(r+1)+n(s+1)=mn}a_{r,s}\xi^{p(r+1)+q(s+1)-qn}$. Hence the Jacobian
matrix at $(\xi,0)$ is \eqref{eq:singp}.
\end{proof}

Using Lemma \ref{singular points} we get:

\begin{teorema}\label{monomial}
Let ${\mathcal F}_W\in Fol(f)$ be a foliation defined by $W$ as in (\ref{quasihomogenea}). Then ${\mathcal F}_W$ is a generalized curve foliation if and only if $\w(\Delta(x,y))\geq mn-n-m$, and $(1+\sum_{ m(r+1)+n(s+1)=mn}a_{r,s}\xi^{p(r+1)+q(s+1)-qn}) \notin  \Q^{-}\cup\{0\}$, where $m(r+1)+n(s+1)=nm$ and $\xi^d=1$.
\end{teorema}

\begin{proof}

Suppose  that the foliation  ${\mathcal F}_W$ is generalized curve. Since its union of separatrices is $\{f(x,y)=0\}$ and  by Proposition \ref{Newton}, we get the equality $\NP(W)=\NP(df)$  and we have $\w(W)\geq mn$. In addition,  using the representation given in (\ref{quasihomogenea})  we have $\w(y\sum a_{i,j}x^iy^j\dd x)\geq mn$ \,\,and \,\, $\w(x\sum a_{i,j}x^iy^j\dd y)\geq mn$; so    $\w(\Delta(x,y))+n+m=n(i+1)+m(j+1)\geq mn$. Moreover the points $(\xi,0)$ are not saddle-nodes, hence by \eqref{eq:singp} we get 
$(1+\sum_{r,s}a_{r,s}\xi^{p(r+1)+q(s+1)-qn}) \notin  \Q^{-}\cup\{0\}$.

On the other hand,  the hypothesis  $\w(\Delta(x,y))\geq mn-n-m$  allows us to apply Lemma \ref{singular points}. Since $(1+\sum_{r,s}a_{r,s}\xi^{p(r+1)+q(s+1)-qn}) \notin \Q^{-}\cup\{0\}$ we conclude that ${\mathcal F}_W\in Fol(f)$ is a generalized curve foliation.
\end{proof}

\begin{ejemplo}
\label{ej:cg}
Suppose that	$W=\dd (y^3-x^6) +axy(3x\dd y-6y)\dd x$, where $a\in \C$. In this case 
	$p=q=1$, and the toric morphism is 
	 $x=uv$, $y=uv^2$. The total transform of $W$ is 
\[\Pi^{*} W:\,\,\,u^2v^5[((3-6u^3)-3au)v\dd u +6(1-u^3)u\dd v]=u^2v^5(A\dd x+B\dd y),\]
where $A=((3-6u^3)-3au)v$ and $B=6(1-u^3)u$. Then the singularities of the strict transform of $W$ are $(0,0),(\xi,0)$ with $\xi^3=1$.
	Now $6(r+1)+3(s+1)=18$ if and only if $(r,s)=(1,1)$. Hence, if  $(1+a\xi)\in \Q^{-}$, for some $ \xi\,\tx{with}\,\,\xi^3=1$, then the foliation ${\mathcal F}_{W}$ is   not generalized curve.
\end{ejemplo}

Mattei and Salem \cite{Mattei-Salem} consider a  family of foliations, more general than the generalized curved foliations, called foliations of the \emph{second type}, where  saddle nodes are admitted in the reduction process, provided that 
they lie in the regular part of the divisor, with their \emph{weak separatrices} (the separatrices associated with the zero eingenvalue) transversal to the divisor.\\

In the next corollary we give the condition in order that a foliation of second type becomes a generalized curve foliation following \cite[Theorem 1.2 (b)]{FS-GB-SM}:

\begin{coro}
\label{monomial second type}
Let ${\mathcal F}_W\in Fol(f)$ be a foliation defined by $W$ as in (\ref{quasihomogenea}). Suppose that ${\mathcal F}_W\in Fol(f)$ is of the second type. Then  ${\mathcal F}_W\in Fol(f)$ is a generalized curve foliation if and only if
\[(1+\sum_{ m(r+1)+n(s+1)=mn}a_{r,s}\xi^{p(r+1)+q(s+1)-qn}) \notin  \Q^{-}\cup\{0\}.\]
\end{coro}

\begin{proof}
It is a consequence of  \cite[Theorem 1.2]{FS-GB-SM} and Theorem \ref{monomial}.
\end{proof}

\begin{nota}
From Theorem \ref{monomial}  we can deduce  Proposition 5.7 of \cite{FS-GB-SM}: let  $n,m$ be two positive integers which are not coprime. Consider $f(x,y)=y^{n}-x^{m}$ and ${\mathcal F}_W\in Fol(f)$, where $W=\dd f +\Delta'(x,y)(nx\dd y -my \dd x)$ with $\Delta'(x,y)\in \mathbb C\{x,y\}$. 
If we suppose that  $i_{0}(\Delta',f)>mn-m-n$ then from the proof of Lemma \ref{singular points} we get that  $\tx{ord}_{v}(S)> 0$. Hence $-d(1+\displaystyle \sum_{r,s}a_{r,s}\xi^{p(r+1)+q(s+1)-qn})=-d$. In particular, the foliation ${\mathcal F}_W$ is generalized curve.\\

In Example \ref{ej:cg} we have a family of nondicritical generalized curve foliations ${\mathcal F}_W$ with  $i_{0}(\Delta',f)=mn-m-n$ when $(1+a\xi)\not\in \Q^{-}$.\\

\end{nota}

\section{Weierstrass form of a 1-form}
\label{sect: Weierstrass}

In this section we introduce, using the Weierstrass division of power series, a  distinguished  equation for a given  1-form, with respect to any Weierstrass polynomial, called {\it Weierstrass form}. The Weierstrass form is very useful for computations and it can be considered as   a generalization, { to any 1-form } of the prenormal form given by Loray  for foliations with monomial separatrix. The Weierstrass form is well-defined for any $1$-form. Nevertheless, in this paper we are using the Weierstrass forms associated with  $1$-form $W$ defining a nondicritical foliation.

\medskip

Let $f=\sum_{i=0}^{n}a_i(x)y^{n-i}\in\mathbb{C}\{x\}[y]$ be a reducible polynomial (not necessarily 
irreducible) of degree $\deg_y(f)=n$, $a_0(0)\neq 0$ and $W\in \Omega^{1}_{\mathbb{C}^2,0}$ be a 1-form.
Here we do not suppose that $C:\{f(x,y)=0\}$ is the union of separatrices of the foliation $\mathcal{F}_{W}$. In that follows we will obtain an equation of $W$ in function of $f$ and $\dd f$.

\begin{lema}
\label{le:preW}
If $W\in \Omega^1_{\mathbb{C}^2,0}$ and 
$f(x,y)=\sum_{i=0}^{n}a_i(x)y^{n-i}\in \C\{x\}[y]$ with $\deg_y(f)=n>1$ and $a_0(0)\neq 0$, then there exist unique $h,p\in\mathbb{C}\{x,y\}$ and $A,B\in\mathbb{C}\{x\}[y]$ with $B=0$ or $\deg_y(B)<n-1$ and $A=0$ or $\deg_y(A)<n$ such that
\begin{equation}
\label{eq:W}
W=h\dd f+pf\dd x+A\dd x+B\dd y.
\end{equation}
\end{lema}
\begin{proof}
Put $W=R(x,y)\dd x+S(x,y)\dd y\in\Omega^1_{\mathbb{C}^2,0}$. After the Weierstrass division of $S$ by $f_y$, there exist $h\in \C\{x,y\}$ and $B\in \C\{x\}[y]$ such that $S=hf_y+B$, where  $B=0$ or $\deg_yB<\deg_yf_y=\deg_yf-1=n-1$ and
$W=R\dd x+hf_y\dd y+B\dd y.$ Since $\dd f=f_x\dd x+f_y\dd y$ then
$W=(R-hf_x)\dd x+h\dd f+B\dd y.$
Now, by the Weierstrass division of  $R-hf_x$ by $f$ there exist $p\in \C\{x,y\}$ and $A\in \C\{x\}[y]$ such that
$R-hf_x=pf+A$ with $A=0$ or $\deg_yA<\deg_yf=n$ and 

\[ W=h\dd f+pf\dd x+A\dd x+B\dd y.\]

Suppose that $h_1\dd f+p_1f \dd x+A_1\dd x+B_1\dd y=h_2 \dd f+p_2f\dd x+A_2\dd x+B_2\dd y$ with $h_i,p_i\in\mathbb{C}\{x,y\}$ and $A_i,B_i\in\mathbb{C}\{x\}[y]$ with $\deg_y(B_i)<n-1$ and $\deg_y(A_i)<n$ for $i=1,2$. Then $(h_1-h_2)f_y=B_2-B_1$ and $(h_1-h_2)f_x+(p_1-p_2)f=A_2-A_1$.

As $\deg_y(B_i)<n-1=\deg_y(f_y)$, we get $h_1=h_2$ and $B_1=B_2$. Similarly, $\deg_y(A_i)<n=\deg_y(f)$ implies that $p_1=p_2$ and $A_1=A_2$.
\end{proof}

\begin{defi}
\label{Wform}
Let $W\in\Omega^1_{\mathbb{C}^2,0}$ and $f=\sum_{i=0}^{n}a_i(x)y^{n-i}\in\mathbb{C}\{x\}[y]$ with $\deg_y(f)=n>1$ and $a_0(0)\neq 0$. The {\bf Weierstrass form} of $W$ with respect to $f$ is $h\dd f+pf\dd x+A\dd x+B\dd y$, where 
$h,p\in\mathbb{C}\{x,y\}$ and $A,B\in\mathbb{C}\{x\}[y]$ are unique and satisfying $B=0$ or $\deg_y(B)<n-1$ and $A=0$ or $\deg_y(A)<n$.
\end{defi}

We can rewrite \eqref{eq:W} as 
\begin{equation}
\label{eq:W1}
W=h\dd f+pf\dd x +\omega,
\end{equation}

where $\omega=\sum A_{i,j}x^{i-1}y^j \dd x+ \sum B_{i,j}x^{i}y^{j-1} \dd y=A \dd x+B\dd y$  for some $A_{i,j},B_{i,j}\in \C$.
Since $B=0$ or $\deg_{y} B<n-1$ then $B_{0,n}=0$.\\

Observe that, in \eqref{eq:W1},  $h$, $p$ and $w$ depend on $f$. Moreover if $W=R(x,y)\dd x+S(x,y)\dd y\in\Omega^1_{\mathbb{C}^2,0}$, where $R(x,y), S(x,y)\in \mathbb{C}\{x\}[y]$ and $f\in\mathbb{C}\{x\}[y]$ with  $\deg_{y}(f)>\max\{\deg_{y}(R),\deg_{y}(S)+1\}$ then, in \eqref{eq:W1}, $h=p=0$ and the Weierstrass form of $W$ is $w$.

\medskip

\begin{prop}  
\label{Fol(f)}
Let ${W}\in [Fol(f)]$ as \eqref{eq:W1}, where $h$ is a unit of $\C\{x,y\}$ and $f(x,y)=\sum_{i=0}^{n}a_i(x)y^{n-i}\in \C\{x\}[y]$ with $n=\deg_{y}(f)=\ord f$ and $a_0(0)\neq 0$.
Then ${\mathcal F}_{W}\in Fol(f)$,  that is $C$ is the union of separatrices of ${\mathcal F}_{W}$. 
\end{prop}
\begin{proof}
Given that $\deg_{y}(f)=\ord f$,  the multiplicity of $W$ equals the multiplicity of $h\dd f+\omega$. Since $h$ is a unit we get $ \mult (W)=\min\{\ord(f_{x}+A),\ord(f_{y}+B)\}\leq \ord(f_{y}+B)$.
But $B_{0,n}=0$ and $\ord f=\deg_{y}(f)$, so $\ord(f_{y}+B) \leq n-1$. Hence 
\begin{equation}
\label{mult}
\mult (W)\leq n-1.
\end{equation}

Suppose that ${\mathcal F}_{W}$  has other separatrix $g(x,y)=0$.  By Theorem \ref{comparando multiplicidades} we have $\mult (W)\geq \ord(fg)-1>n-1,$ which is a contradiction after the inequality \eqref{mult}. 

\end{proof}

\medskip

In Section \ref{Char} we will see how the Weierstrass forms allow us to give new characterizations of generalized curve foliations. We can also characterize the second type foliations using the Weierstrass forms. To do this, we first remember
the characterization  given by Mattei and Salem:

\begin{teorema}(\cite[Th\'eor\`eme 3.1.9]{Mattei-Salem})
\label{secondtype}
Let $W=A(x,y)\dd x+B(x,y)\dd y$  be a 1-form. Suppose that  ${\mathcal F}_W$ is a non dicritical foliation and $C:\{f(x,y)=0\}$  is its union of separatrices. Then
 ${\mathcal F}_W$ is a second type foliation  if and only if $\mult(W)= \mult(\dd f)$.
\end{teorema}

Let $f(x,y)=\sum_{i=0}^{n}a_i(x)y^{n-i}\in \C\{x\}[y]$ with $n=\deg_{y}(f)=\ord f$ and $a_0(0)\neq 0$. Consider a $1$-form $W \in \Omega^{1}_{\mathbb{C}^2,0}$ in the Weierstrass form as \eqref{eq:W1}, where $\omega =A(x,y)\dd x+B(x,y)\dd y$. Let us relate some algebraic aspects of the $1$-forms $W$, $\omega$ and the {\it parameters} $h$ and $p$ in the Weierstrass form. 

Notice that $\frac{W\wedge \dd f}{\dd x\wedge \dd y}\in(f)$ if and only if $\frac{\omega\wedge \dd f}{\dd x\wedge \dd y}\in(f)$. More specifically, if $g$ is the cofactor of $W$, that is, $W\wedge \dd f=gf\dd x\wedge \dd y$, then the cofactor of $\omega$ is $g-pf_y$.

In addition, remark that:
\begin{itemize}
	\item If $h(0,0)=0$ then $\mult(W)=\mult(\dd f)$ if and only if $\mult(\omega)=\mult(\dd f).$
		
	\item If $h(0,0)\neq 0$ then $\mult(W)=\mult(\dd f)$ if and only if $\mult(\omega)\geq \mult(\dd f).$
\end{itemize}

As, by Theorem \ref{secondtype}, the equality $\mult(W)=\mult(\dd f)$ characterizes the $1$-forms that define second type foliations, we can read this property using $\omega$.

\medskip

In \eqref{eq:W1} we remarked that for any Weierstrass form $W=h\dd f+pf\dd x+\omega,$ 
	where $\omega=\sum A_{i,j}x^{i-1}y^j \dd x+ \sum B_{i,j}x^{i}y^{j-1} \dd y=A \dd x+B\dd y$  for some $A_{i,j},B_{i,j}\in \C$, we get $B_{0,n}=0$.
	
	If $f\in K(n,m)$ and $W\in[Fol(f)]$ we can obtain additional information about the coeficient $A$. 
	
	We can write,  without lost of generality
\begin{equation}
\label{eq:g1}
f(x,y)=y^n-x^m+\sum_{in+jm>nm}a_{i,j}x^iy^j.
\end{equation} 

\begin{lema}
\label{coef:0}
Let  $f\in K(n,m)$. If $W\in [Fol(f)]$, with $W$ as in \eqref{eq:W1}, then $A_{m,0}=0$. 
\end{lema}
\begin{proof}
Consider  $f(x,y)$ as in \eqref{eq:g1}. Since $C:\{f(x,y)=0\}$ is the separatrix of ${\mathcal F}_W$ and $\omega=W-h\dd f- p f\dd x$  then $f(x,y)=0$ is also a separatrix of $w$. Consider a parametrization of $f(x,y)=0$ given by $(x(t),y(t))=(t^n,\, t^m+\cdots)$, where $\cdots$ means terms of greater degree.  Therefore $\sum A_{i,j}x(t)^{i-1}y(t)^j \dd x(t)+ \sum B_{i,j}x(t)^iy(t)^{j-1} \dd y(t) =0$.  Suppose that $A_{m,0}\neq 0$. Since $A_{m,0}x(t)^{m-1}\dd x(t)=nA_{m,0}t^{nm-1}$ then  there is $(i_{0},j_{0})\in \supp{\mathcal A}
\cup  \supp{\mathcal B}$ such that $ni_{0}+mj_{0}=nm$, where ${\mathcal A}=\sum A_{i,j}x(t)^{i-1}y(t)^j \dd x(t)-nA_{m,0}t^{nm-1} $ and ${\mathcal B}=\sum B_{i,j}x(t)^{i}y(t)^{j-1}\dd y(t)$. But $n$ and $m$ are coprime, so
$(i_{0},j_{0})\in \{(m,0),(0,n)\}$, which is a contradiction since these  points are not in 
$\supp{\mathcal A}\cup  \supp{\mathcal B}.$
\end{proof}

\section{Characterization of generalized curve foliations: general case}
\label{Char}

In this section we present our  main result. For that we need the notion of  \emph{GSV-index}.

Let $f(x,y)\in \C\{x,y\}$ and $W \in [Fol(f)]$. By \cite[page 198]{Lins} in the irreducible case and \cite[(1.1) Lemma]{Suwa} in the reduced case, there are $g,k\in \C\{x,y\}$ and a 1-form $\eta$ such that $gW=k\dd f+f\eta$, with $f$ and $k$ coprime. 

\begin{defi}
\label{def:GSV}
With the above notations, the GSV-index of  $W$ with respect to $C:\{f(x,y)=0\}$ is
\[
GSV(W,C):=\frac{1}{2\pi i}\int_{\partial C}\frac{g}{k}\dd \left(\frac{k}{g}\right).
\]
\end{defi}
Suppose that $f(x,y)$ is irreducible and consider a parametrization $(x(t),y(t))$ of $C$. Write $W=A(x,y)\dd x+ B(x,y)\dd y$ and $\eta=p(x,y)\dd x+ q(x,y)\dd y$. We have 
\[
W=\left(\frac{k}{g}f_{x}+\frac{f}{g}p\right)\dd x+ \left(\frac{k}{g}f_{y}+\frac{f}{g}q\right)\dd y.
\]

On the other hand, 
\[
\frac{k}{g}(x(t),y(t))=\frac{\left(\frac{k}{g}f_{y}+\frac{f}{g}q\right)}{f_{y}}(x(t),y(t)),
\]

so the number of zeroes of $k$ restricted to $C$ minus the number of zeroes of $g$ restricted to $C$ equals $\ord_{t}\frac{B}{f_{y}}(x(t),y(t))$. Hence by Rouch\'e-Hurwitz theorem we have
\begin{equation}
\label{GSV}
GSV(W,C)=\ord_{t}\frac{B}{f_{y}}(x(t),y(t)).
\end{equation}

A similar calculation as before shows that 

\begin{equation*}
\label{GSV2}
GSV(W,C)=\ord_{t}\frac{A}{f_{x}}(x(t),y(t)).
\end{equation*}

Now, consider $C:\{f(x,y)=0\}$ for $f=f_{1}\cdot f_{2}$ with $f_{1}$ and $f_{2}$ irreducible. After \cite[page 532]{Brunella} we have 
\begin{equation}
\label{GSV reducida}
GSV(W,C)=GSV(W,C_{1})+GSV(W,C_{2})-2i_{0}(f_{1},f_{2}),
\end{equation}

where $C_{i}:\{f_{i}(x,y)=0\}.$ The equality \eqref{GSV reducida} is also true when $f_{1},f_{2}$ are reduced, not necessary irreducible.

\medskip

Cavalier and Lehmann gave a characterization of generalized curve foliations using the $GSV$-index:

\begin{teorema}(\cite[Th\'eor\`eme 3.3]{Cav-Le})
\label{th:Cav-Leh}
Let $C:\{f(x,y)=0\}$ be a reduced curve  and ${\mathcal F}_W \in Fol(f)$ a nondicritical foliation. Then  ${\mathcal F}_W$ is generalized curve if and only if $GSV(W,C)=0$.
\end{teorema}

\medskip

In the remainder of the section we will present our main result and consequences of it.

\medskip

Let $f\in \C\{x\}[y]$, where $f=f_1\cdots f_r$ is the factorization of $f$ into irreducible factors. Consider  $C:\{f(x,y)=0\}$.

\begin{teorema}\label{general}
Let ${\mathcal F}_W\in Fol(f)$ and 
 $h\dd f+ pf \dd x+ A \dd x+ B\dd y=\mathcal{A} \dd x + \mathcal{B} \dd y$ the Weierstrass form of $W$ with respect to $f$. Then
${\mathcal F}_W $ is a generalized curve foliation if and only if $h \in \C\{x,y\}$ is a unit and $i_{0}(\mathcal{B}, f)=\mu(f)+i_0(f,x)-1.$
\end{teorema}
\begin{proof} First, observe that $\mathcal{B}=hf_{y}+B$. Hence, 
\[i_{0}(\mathcal{B}, f)=i_{0}(hf_{y}+B, f)=\sum_{j=1}^{r}i_{0}\Big(h\prod_{i\neq j} f_i(f_j)_y+B,f_j \Big).\]
Suppose that  $\sum_{j=1}^{r}i_{0}\big(h\prod_{i\neq j} f_i(f_j)_y+B,f_j \big)=\mu(f)+i_0(f,x)-1$ and $h(x,y)$ is a unit.\\ If $C_i:\{f_i(x,y)=0\}$ and $C:\{f_1\cdots f_r=0\}$
 then, by \eqref{GSV reducida}, the  $GSV$- index of $W$ with respect to $C$ is 
 
 \[
 GSV(W,C) = \sum_{i=1}^{r}GSV(W,C_i)-2\sum_{1\leq i<j\leq r}i_{0}(f_i,f_j).
 \]

By hypothesis, 
\[
W=\Big(h\sum _{i=1}^{r} \prod_{j\neq i}f_j(f_i)_x+fp+A\Big)\dd x+
\Big(h\sum_{i=1}^{r}\prod_{j\neq i} f_j(f_i)_y+B\Big)\dd y=\mathcal{A}\dd x+\mathcal{B}\dd y,
\]
hence $\mathcal{A}=h\sum _{i=1}^{r} \prod_{j\neq i}f_j(f_i)_x+fp+A$ and $\mathcal{B}=h\sum_{i=1}^{r}\prod_{j\neq i} f_j(f_i)_y+B$.
Then by \eqref{GSV} the $GSV$-index of  $W$ with respect to the separatrix $C_j:\{f_j(x,y)=0\}$ is

\[
 GSV(W,C_{j}) =\ord_{t} \left(\frac{{\mathcal B}(x_{j}(t),y_{j}(t))}{(f_{j})_{y}(x_{j}(t),y_{j}(t))}\right)=i_{0}({\mathcal B},f_{j})-i_{0}((f_j)_y,f_{j}),
\]
where $(x_{j}(t),y_{j}(t))$ is a parametrization of $C_{j}$. By definition of ${\mathcal B}$ and using \eqref{LTeissier} we have
\[
 GSV(W,C_{j}) =i_{0}\Big(h\prod_{i\neq j} f_i(f_j)_y+B,f_j \Big)-\Big(\mu(f_{j})+i_0(f_{j},x)-1\Big).
\]

Hence
\[
 GSV(W,C) =\sum_{j=1}^{r}i_{0}\Big(h\prod_{i\neq j} f_i(f_j)_y+B,f_j \Big)-\sum_{j=1}^{r}\Big(\mu(f_{j})+i_0(f_{j},x)-1\Big)-2\sum_{1\leq i<j\leq r}i_{0}(f_i,f_j).
 \]
 Now, by the formula of Milnor for a reduced curve (see for example \cite[Theorem 6.5.1]{Wall}) we obtain
 \begin{equation}
 \label{eq:GSV reducida}
 GSV(W,C) =\sum_{j=1}^{r}i_{0}\Big(h\prod_{i\neq j} f_i(f_j)_y+B,f_j \Big)-\Big(\mu(f)+i_0(f,x)-1\Big).
 \end{equation}

So, by Theorem \ref{th:Cav-Leh}
we conclude that   ${\mathcal F}_W$ is a generalized curve foliation.\\

Suppose now that ${\mathcal F}_W$ is a generalized curve foliation which   union of separatrices is $C:\{f(x,y)=0\}$. Hence ${\mathcal NP}(W)={\mathcal NP}(f)$. Moreover, since $h\dd f+ pf \dd x+ A \dd x+ B\dd y$ is the Weierstrass form of $W$ with respect to $f$  we get 
$B_{0,n}=0$ and we conclude that $h$ is a unit. In particular, we have again the equality \eqref{eq:GSV reducida}.

Finally, since ${\mathcal F}_W$ is a generalized curve foliation, again by Theorem \ref{th:Cav-Leh}, we get $0=GSV(W,C)$. This finishes the proof.
\end{proof}

\medskip

Theorem \ref{general} gives a characterization of a generalized curve foliation  ${\mathcal F}_W$, where $W=\mathcal{A} \dd x + \mathcal{B} \dd y$, using the polar $\mathcal{B}$ of the foliation ${\mathcal F}_W$ and the {\em polar} $f_{y}$ of the separatrix $C:\{f(x,y)=0\}$.

\begin{coro}
\label{cor:polar separatrix}
Let ${\mathcal F}_W\in Fol(f)$ and 
 $h\dd f+ pf \dd x+ A \dd x+ B\dd y=\mathcal{A} \dd x + \mathcal{B} \dd y$ the Weierstrass form of $W$ with respect to $f$. Then
${\mathcal F}_W $ is a generalized curve foliation if and only if $h \in \C\{x,y\}$ is a unit and $i_{0}(\mathcal{B}, f)=i_{0}(f_{y},f).$ In particular $i_{0}(B, f)\geq i_{0}(f_{y},f).$
\end{coro}

Suppose now that $f\in\mathbb{C}\{x\}[y]$ is an irreducible monic  polynomial  of degree $n=\ord f$ with semigroup  $\Gamma(f) = \langle v_0,v_1,v_2,\ldots ,v_g \rangle$ ($v_{0}=n$). Remember that we denote $e_{i}=\gcd(v_{0} ,\ldots,v_{i})$ for $i\in \{0,\ldots, g\}$ and $n_{i}=	\frac{e_{i-1}}{e_{i}}$ for $i\in\{1,\ldots, g\}$. By convention we put $n_{0}=1$. We say that $f_k\in\mathbb{C}\{x\}[y]$ is a $k$-{\bf semiroot} of the polynomial $f$ if $f_k$ is monic, $\deg_y(f_k)=\frac{v_0}{e_{k-1}}=n_0n_1\cdots n_{k-1}$  and $i_0(f_k,f)=v_k$. The notion of semi-root is a generalization of the  {\bf characteristic approximate roots}  introduced and studied by Abhyankar and Moh in  \cite{A-M}.

Applying \cite[Corollary 5.4]{P-P} to  the polynomial  $B\in \C\{x\}[y]$ in \eqref{eq:W}, it can be uniquely written as a finite sum of the form
\begin{equation}
\label{expansion}
B=\sum_{\hbox{\rm finite}}a_{\alpha}(x)f_1^{\alpha_1}\cdots f_g^{\alpha_g},
\end{equation}

where $a_{\alpha}(x)\in\mathbb{C}\{x\}$, $f_{k}$ are $k$-semiroots of $f$ and $0\leq \alpha_k<n_k$. Since $\deg_yB<\deg_yf-1$ the polynomial $f$ does not appear as a factor in the terms of the right-hand side of \eqref{expansion}.

\medskip

As a $k$-semiroot $f_k$ is irreducible and admits semigroup $\langle \frac{v_0}{e_{k-1}}=\frac{n}{e_{k-1}},\frac{v_1}{e_{k-1}}=\frac{m}{e_{k-1}},\ldots ,\frac{v_{k-1}}{e_{k-1}}\rangle$ it follows that its Newton polygon has a single compact face with vertices $\left(0,\frac{m}{e_{k-1}}\right)$ and $\left(\frac{n}{e_{k-1}},0\right)$ and consequently $v_{n,m}(f_k)=\frac{nm}{e_{k-1}}$. \\

In this way, using the relations \eqref{exponentes Puiseux} we get
\begin{equation}
\label{*}
i_{0}(f,f_k)=v_k=\frac{nm}{e_{k-1}}+\sum_{i=2}^{k}\frac{n_i\cdots n_k}{n_{k}}(\beta_i-\beta_{i-1})=v_{n,m}(f_k)+\sum_{i=2}^{k}\frac{e_{i-1}}{e_{k-1}}(\beta_i-\beta_{i-1}).
\end{equation}
In particular, $i_0(f,f_k)\geq v_{n,m}(f_k)$ with equality if and only if $k=1$.

\begin{nota}
\label{distintos}
For any two distinct terms $T_{\alpha}:=a_{\alpha}(x)f_1^{\alpha_1}\cdots f_g^{\alpha_g}$ and $T_{\alpha'}:=a_{\alpha'}(x)f_1^{\alpha'_1}\cdots f_g^{\alpha'_g}$ of \eqref{expansion} we get $i_0(T_{\alpha},f)\neq i_0(T_{\alpha'},f)$.
\end{nota}

In addition, remark that we can not have $\alpha_i=n_i-1$ for all $i=1,\ldots ,g$. Indeed, if this is the case we get
\begin{eqnarray*}
\deg_{y}(B)&=&\sum_{i=1}^{g}(n_i-1)\deg_{y}(f_i)=\sum_{i=1}^{g}(n_i-1)n_0n_1 \cdots n_{i-1}\\
 &=&\sum_{i=1}^{g}n_0n_1 \cdots n_{i-1}n_i-\sum_{i=1}^{g}n_0n_1 \cdots n_{i-1}=n_0n_1\cdots n_g-n_0=v_0-1=n-1,
 \end{eqnarray*}
which is a contradiction.\\

As a consequence of $\eqref{*}$ and Remark \ref{distintos} we have that $i_0(f,B)\geq v_{n,m}(f)$ with equality if and only if $f\in K(n,m)$.

\begin{lema}\label{difiere}
With the above notations, if $f$ is irreducible then $i_0(B,f)\neq i_0(f_y,f)$.
\end{lema}

\begin{proof}
Suppose that $i_0(f_y,f)=i_0(B,f)$. After Remark \ref{distintos}, there exists a unique $g$-tuple 
$(\alpha_1,\ldots ,\alpha_g)$ with $0\leq \alpha_i<n_i$ such that
\[i_0(B,f)=i_0(a_{\alpha}(x)f_1^{\alpha_1}\cdot\ldots\cdot f_g^{\alpha_g},f)=\sum_{i=1}^{g}\alpha_iv_i+\lambda_0v_0,\]

where $\lambda_0=\ord_{x}a_{\alpha}(x)$. Now, by \eqref{LTeissier} we get $i_0(f_y,f)=\mu(f)+v_0-1=\sum_{i=1}^{g}(n_i-1)v_i$ (see for example \cite[Proposition 7.5 (ii), page 102]{Hefez} for the last equality). In this way, we have
\[\sum_{i=1}^{g}(n_i-1)v_i=i_0(f_y,f)=i_0(B,f)=\sum_{i=1}^{g}\alpha_iv_i+\lambda_0v_0,\]
that is $\sum_{i=1}^{g}(n_i-1-\alpha_{i})v_i-\lambda_0v_0=0$.
But, this implies  that $\lambda_0=0$ and $\alpha_i=n_i-1$ for all $i=1,\ldots ,g$, which is a contradiction. 
Hence, $i_0(f_y,f)\neq i_0(B,f)$.
\end{proof}

\begin{nota}
Observe that Lemma \ref{difiere} is not true for $f$ reduced (non-irreducible): consider $f(x,y)=y^{2}-x^{2}$ and $B=x^{e}$. We have $\deg_{y}(B)=0<1=\deg_{y}(f)-1$, $i_0(f_y,f)=2$ and $i_0(B,f)=2e$. So, for $e=1$ we get $i_0(f_y,f)=i_0(B,f)$ and $i_0(f_y,f)\neq i_0(B,f)$ for $e\neq 1$.
\end{nota}

\begin{coro}
\label{cor: desig}
With the above notations, for $f$ irreducible we have, $i_{0}(B,f)>i_{0}(f_{y},f)$ if and only if  $i_{0}(f_{y}+B,f)=\mu(f)+i_0(f,x)-1$.
\end{coro}
\begin{proof}
Suppose that $i_{0}(B,f)>i_{0}(f_{y},f)$. Then 
\[i_{0}(f_{y}+B,f)=\min\{i_{0}(f_{y},f),i_{0}(B,f)\}=i_{0}(f_{y},f),\]
 
and by Teissier's Lemma (see \eqref{LTeissier}) we have $i_{0}(f_{y}+B,f)=\mu(f)+i_0(f,x)-1$.

Now we suppose that $i_{0}(f_{y}+B,f)=\mu(f)+i_0(f,x)-1$, that is, by Teissier's Lemma,
 $i_{0}(f_{y}+B,f)=i_0(f,f_{y})$ and by Lemma \ref{difiere} we conclude $i_{0}(B,f)>i_{0}(f_{y},f)$.
\end{proof}

\begin{nota}
\label{**}
If $W$ is written as \eqref{eq:W} and $W\in [Fol(f)]$ then $A(x,y)\dd x+ B(x,y) \dd y \in [Fol(f)]$. We must have $i_{0}(A,f)+v_{0}=i_{0}(B,f)+v_{1}$. Hence by \eqref{LTeissier} and \eqref{LTeissier2}  we get $i_{0}(A,f)-i_{0}(f_{x},f)=i_{0}(B,f)-i_{0}(f_{y},f)$. So, $i_{0}(B, f) > i_{0}(f_{y}, f)$ if and only if, $i_{0}(A, f) > i_{0}(f_{x}, f)$.
Moreover, by Lemma \ref{difiere}, we have  $i_{0}(A, f)\neq i_{0}(f_{x}, f)$. So,
$i_{0}(f_{x} +A,f)=\min\{i_{0}(f_{x},f),i_{0}(A,f)\}$.
Consequently $i_0(f_y+B,f)=\mu(f)+i_0(f,x)-1$ if and only if $i_0(f_x+A,f)=\mu(f)+i_0(f,y)-1$.
\end{nota}

Next proposition gives us a characterization of generalized curve foliations with a single se\-paratrix.

\begin{prop}
\label{prop:irred}
 Let $f(x,y)\in \C\{x\}[y]$ be irreducible, ${\mathcal F}_W \in Fol(f)$, where
$h\dd f+ pf \dd x+ A \dd x+ B\dd y$ is the Weierstrass form of $W$ with respect to $f$. Then ${\mathcal F}_W$  is a generalized curve foliation if and only if  $h \in \C\{x,y\}$ is a unit and 
$i_0(B,f)>i_0(f_y,f)$.
\end{prop}
\begin{proof}
It is a consequence of Corollaries \ref{cor:polar separatrix} and \ref{cor: desig}.
\end{proof}

Next corollary gives us a characterization of generalized curve foliations with a single separatrix of genus $1$ in terms of the weighted order (see Definition \ref{worder}).

\begin{coro}
\label{coro:genus1}
Let  $f\in K(n,m)$, ${\mathcal F}_W \in Fol(f)$, where
$h\dd f+ pf \dd x+ A \dd x+ B\dd y$ is the Weierstrass form of $W$ with respect to $f$.  Then ${\mathcal F}_W$ is 
a generalized curve foliation if and only if  $\upsilon_{n,m}(\omega) > nm$,
where $\omega=A \dd x+ B\dd y$.
\end{coro}
\begin{proof}

By Remark \ref{**} the equality $i_0(f_y+B,f)=\mu(f)+i_0(f,x)-1$ is equivalent to $i_0(f_x+A,f)=\mu(f)+i_0(f,y)-1$. It follows, by Corollary \ref{cor: desig}, that this is equivalent to claim $i_0(f_y,f)<i_0(B,f)$ and $i_0(f_x,f)<i_0(A,f)$. As $f\in K(n,m)$, we have $\mu(f)=nm-n-m+1$, $i_0(f_y,f)=nm-m$ and $i_0(f_x,f)=nm-n$ (see \eqref{LTeissier} 
and \eqref{LTeissier2}). So,  $nm<i_0(B,f)+m$ and $nm<i_0(A,f)+n$. Since $i_0(H,f)=v_{n,m}(H)$, for any $H\in\mathbb{C}\{x,y\}$ then $i_0(f_y+B,f)=\mu(f)+i_0(f,x)-1$ is equivalent to
$nm<\min\{v_{n,m}(B)+m,v_{n,m}(A)+n\}=v_{n,m}(\omega)$.We finish the proof using Proposition \ref{prop:irred} and Corollary \ref{cor: desig}.
\end{proof}

\begin{ejemplo}
Let $W_c = \dd(y^3-x^4)+cx^2y(3x \dd y-4y\dd x)$, $c\in\C^{*}.$	
 By  Proposition \ref{Fol(f)} we get ${\mathcal F}_{W_c}\in Fol(y^3-x^4)$. In this case $B(x,y)= -4cx^2y^2$ and $i_{0}(B,f)=14>\mu(f)+i_0(f,x)-1=9$, then ${\mathcal F}_{W_c}$ is a generalized curve foliation.		
\end{ejemplo}

 {\small  Evelia Rosa Garc\'{\i}a Barroso\\
Departamento de Matem\'aticas, Estad\'{\i}stica e I.O. \\
Secci\'on de Matem\'aticas, Universidad de La Laguna\\
Apartado de Correos 456\\
38200 La Laguna, Tenerife, Espa\~na\\
e-mail: ergarcia@ull.es  \\ORCID ID: 0000-0001-7575-2619}

\medskip

{\small Marcelo Escudeiro Hernandes\\
Departamento de Matem\'atica\\
Universidade Estadual de Maring\'a\\
Avenida Colombo 5790\\
Maring\'a, Maring\'a-PR 87020-900.\\
Brazil\\
e-mail: mehernandes@uem.br   \\ORCID ID: 0000-0003-0441-8503}

\medskip

{\small M. Fernando Hern\'andez Iglesias\\
Facultad de Ciencias Matem\'aticas, Escuela de Matem\'atica, \\
Universidad Nacional Mayor de San Marcos, \\
Cercado de Lima 15081, Peru\\
e-mail: {mhernandezi@unmsm.edu.pe}\\
ORCID ID: 0000-0003-0026-157X}

\end{document}